\UseRawInputEncoding
\documentclass[12pt]{amsart}

\usepackage{setspace, amsmath, amsthm, amssymb, amsfonts, amscd, epic, graphicx, ulem, dsfont}
\usepackage[T1]{fontenc}
\usepackage{multirow}
\usepackage{bbm}

\makeatletter \@namedef{subjclassname@2010}{
  \textup{2020} Mathematics Subject Classification}
\makeatother

\newtheorem{thm}{Theorem}[section]
\newtheorem{lem}{Lemma}[section]
\newtheorem{pro}{Proposition}[section]
\newtheorem{cor}{Corollary}[section]

\theoremstyle{remark}
\newtheorem*{rema}{Remark}

\newtheorem{exa}{\textbf{Example}}

\theoremstyle{definition}

\newcommand{\R}{\mathbb{R}}

\newcommand{\N}{\mathbb{N}}

\begin{document}

\title[]{The sandwich rule for sequences of self-adjoint operators and some applications}
\author[Ch. Chellali and M. H. Mortad]{Chérifa Chellali and Mohammed Hichem Mortad$^*$}

\thanks{$^*$ Corresponding author, who is partially supported by "Laboratoire d'Analyse Mathématique et Applications"}

\dedicatory{}
\thanks{}
\date{}


\address{(The first author) The Higher School of Economics Oran, Algeria.}

\email{chchellali@gmail.com.}

\address{(The corresponding and second author) Laboratoire d'analyse mathématique et applications. Department of
Mathematics, University of Oran 1, Ahmed Ben Bella, B.P. 1524, El
Menouar, Oran 31000, Algeria.}

\email{mhmortad@gmail.com,  mortad.hichem@univ-oran1.dz.}

\begin{abstract}
In this paper, we mainly deal with sequences  of bounded linear
operators on Hilbert space. The main result is the so-called squeeze
theorem (or sandwich rule) for convergent sequences of self-adjoint
operators. We show that this theorem remains valid for all three
main topologies on $B(H)$.
\end{abstract}

\maketitle

\section{Introduction}

Let $H$ be a complex Hilbert space and let $B(H)$ be the algebra of
all bounded linear operators defined from $H$ into $H$.

We say that $A\in B(H)$ is positive, and we write $A\geq 0$, if
$\langle Ax,x\rangle\geq 0$ for all $x\in H$ (this implies that $A$
is self-adjoint, i.e. $A=A^*$). When $A,B\in B(H)$ are both
self-adjoint, then we write $A\geq B$ if $A-B\geq0$.

Recall that if $A$ is a positive operator, then it possesses one and
only one positive square root $B$, i.e. $A=B^2$. We write in this
case $B:=\sqrt A$. Since $A^*A$ is always positive, its unique
positive square root is called its modulus or its absolute value,
and it is designated by $|A|$ (see e.g. \cite{Boucif-Dehimi-Mortad}
and \cite{mortad-Abs-value-BD} for some of its properties). It is
known that

\[-|A|\leq A\leq |A|\]

for any self-adjoint operator $A\in B(H)$. These inequalities follow
from Reid's inequality. See e.g. Exercise 6.3.7 in
\cite{Mortad-Oper-TH-BOOK-WSPC} for a proof and
\cite{Dehimi-Mortad-REID} for a general version of this useful
inequality. Notice that if $A,B\in B(H)$ are self-adjoint, then
$|A|\leq B$ does entail $-B\leq A\leq B$, whilst the reverse
implication need not hold (see \cite{Mortad-counterexamples book OP
TH} for a counterexample). Details about the former claim may be
consulted on Page 144 of \cite{Mortad-Oper-TH-BOOK-WSPC}.

Recall now briefly the main three important types of convergence on
$B(H)$. Let $(A_n)$ be a sequence in $B(H)$.

\begin{enumerate}
  \item We say that $(A_n)$ converges in norm (or uniformly) to $A\in B(H)$ if
  \[\lim_{n\rightarrow\infty}\|A_n-A\|=0.\]
  \item We say that $(A_n)$ strongly converges to $A\in B(H)$ if
  \[\lim_{n\rightarrow\infty}\|(A_n-A)x\|=0\]
  for each $x\in H$.
  \item We say that $(A_n)$ weakly converges to $A\in B(H)$ if
  \[\lim_{n\rightarrow\infty}\langle A_nx,y\rangle=\langle Ax,y\rangle\]
  for each $x,y\in H$.
\end{enumerate}

\begin{rema}
In the case of a complex Hilbert space $H$ (which is our case), or
in the case $(A_n)$ is a sequence of self-adjoint operators (which
is still our case!) even when $H$ is an $\R$-Hilbert space, then
$(A_n)$ weakly converges to $A\in B(H)$ if
  \[\lim_{n\rightarrow\infty}\langle A_nx,x\rangle=\langle Ax,x\rangle\]
  for each $x\in H$ (see e.g. Page 25 in \cite{Kubrusly-book-operatopr-exercise-sol}).
\end{rema}

It is well known to readers that these types of convergence all
coincide when $\dim H<\infty$, and that the weak convergence is the
weakest of the three while the uniform convergence is the strongest
of the three.

Recall also that if $(A_n)$ and $(B_n)$ are two sequences in $B(H)$,
and $A_n$ converges strongly (resp. in norm) to $A$ and $B_n$
converges strongly (resp. in norm) to $B$, then $A_nB_n$ converges
strongly (resp. in norm) to $AB$. Also, if $A_n\geq0$ for all $n$,
and $A_n$ converges strongly (resp. in norm) to $A$, then
$\sqrt{A_n}$ converges strongly (resp. in norm) to $\sqrt{A}$ (see
e.g. Exercise 14 on Page 217 in \cite{RS1}).

Let us include a fairly simple proof (cf. the one on Page 86 in
\cite{Simon-OPER-TH-BOOK-2015}) of the previous result in the case
of the convergence in norm as it might not be very well known to
some readers. The key idea is to first show the following auxiliary
result:
\begin{lem}
If $B,C\in B(H)$ are positive, then
\[\|\sqrt B-\sqrt C\|\leq \sqrt{\|B-C\|}\]
\end{lem}

\begin{proof}
Clearly, $B-C$ is self-adjoint and so (cf. Exercise 5.3.11 in
\cite{Mortad-Oper-TH-BOOK-WSPC})
\[B-C\leq \|B-C\|I\]
or $B\leq C+\|B-C\|I$. Upon passing to the positive square root, we
obtain
\[\sqrt B\leq \sqrt{C+\|B-C\|I}\leq \sqrt C+\sqrt{\|B-C\|I}\]
(also, by a glance at Exercise 5.3.31 in
\cite{Mortad-Oper-TH-BOOK-WSPC}). Therefore,
\[\sqrt B-\sqrt C\leq\sqrt{\|B-C\|}I.\]
By inverting the roles of $B$ and $C$, we obtain
\[\sqrt C-\sqrt B\leq\sqrt{\|B-C\|}I,\]
and hence we get the desired inequality (using again Exercise 5.3.11
in \cite{Mortad-Oper-TH-BOOK-WSPC}).
\end{proof}

\begin{cor}
If $A_n\geq0$ for all $n$, and $A_n$ converges in norm to $A$, then
$\sqrt{A_n}$ converges in norm to $\sqrt{A}$.
\end{cor}

More properties may be consulted in \cite{halmos-book-1982} or
\cite{Kubrusly-book-operatopr-exercise-sol}. Readers should also
consult \cite{Kubrusly-2011-COURS+EXO-FAT-BOOK},
\cite{Mortad-Oper-TH-BOOK-WSPC}, \cite{Mortad-counterexamples book
OP TH}, and \cite{Weidmann} for any other results which will be used
here, and for further reading.

\section{A Squeeze (Sandwich) Rule for Sequences of Self-adjoint Bounded Linear Operators}

The first result is quite simple.

\begin{pro}
Let $(A_n)$, $(B_n)$ and $(C_n)$ be three sequences of self-adjoint
operators in $B(H)$ such that $C_n\leq A_n\leq B_n$ for all
$n\in\N$. If $(B_n)$ and $(C_n)$ converge weakly to some $L\in
B(H)$, then $(A_n)$ too converges weakly to $L\in B(H)$.
\end{pro}

\begin{proof}
\textit{Ex hypothesi}, $C_n\leq A_n\leq B_n$, and so
\[\langle C_nx,x\rangle \leq \langle A_nx,x\rangle\leq \langle B_nx,x\rangle\]
for all $n$ (and all $x\in H$). Since $\langle C_nx,x\rangle$,
$\langle A_nx,x\rangle$ and $\langle B_nx,x\rangle$ are all real
sequences, it ensues by the basic squeeze theorem that $(A_n)$ too
converges weakly to $L\in B(H)$.
\end{proof}

It is therefore more interesting to study the sandwich rule in the
case of strong and uniform convergences.

\begin{thm}\label{sandwich rule S.A. oper THM}
Let $(A_n)$, $(B_n)$ and $(C_n)$ be three sequences of self-adjoint
operators in $B(H)$ such that $C_n\leq A_n\leq B_n$ for all
$n\in\N$. If $(B_n)$ and $(C_n)$ converge in norm (resp. strongly)
to some $L\in B(H)$, then $(A_n)$ too converges in norm (resp.
strongly) to $L\in B(H)$.
\end{thm}

\begin{proof}The proof is split into several parts.
\begin{enumerate}
  \item \textit{The case $L=0$ and $C_n\geq0$ for all $n$:} Assume that $C_n\geq0$ for all
  $n$ and that $L=0$. Hence
  \begin{equation}\label{INEQ 1}
  0\leq \langle C_nx,x\rangle \leq \langle A_nx,x\rangle\leq \langle
  B_nx,x\rangle\end{equation}
  for all $n$ and all $x\in H$. By invoking positive square roots as well as their self-adjointness, it ensues that
  \[\|\sqrt{C_n}x\|^2\leq \|\sqrt{A_n}x\|^2\leq \|\sqrt{B_n}x\|^2\]
  for all $x\in H$ and all $n$ in $\N$. Since $(B_n)$ and $(C_n)$ converge strongly
to $0\in B(H)$, $(\sqrt{B_n})$ and $(\sqrt{C_n})$ too converge
strongly to $0$. Therefore, $(\sqrt{A_n})$
  converges strongly to 0, whereby  $(A_n)$ converges
  strongly to 0. This settles the question of strong convergence.

  To deal with the uniform convergence, consider again Inequalities \ref{INEQ 1}, then
  pass to the supremum over the unit sphere. In other words,
  \[0\leq \sup_{\|x\|=1}\langle C_nx,x\rangle \leq \sup_{\|x\|=1}\langle A_nx,x\rangle\leq \sup_{\|x\|=1}\langle
  B_nx,x\rangle,\]
  that is,
  \[\|C_n\|\leq \|A_n\|\leq \|B_n\|\]
  for each $n\in\N$. It is, therefore, patent that if $(B_n)$ and $(C_n)$ converge in norm
to 0, then so does the sequence $(A_n)$.
  \item \textit{The case $L=0$ with an arbitrary $C_n$:} Consider
  the sequence of absolute values of $(C_n)$, i.e. $|C_n|$. Then, as
  alluded above
  \[C_n+|C_n|\geq 0\]
  for every $n\in\N$. Hence
  \[0\leq C_n+|C_n|\leq A_n+|C_n|\leq B_n+|C_n|\] for all
$n\in\N$. Now, remember that if $C_n$ converges to 0 strongly, then
so does $|C_n|$. In other terms, we have gone back to the previous
case of the proof. Whence, $A_n+|C_n|$ converges to 0, thereby $A_n$
converges strongly to 0, as needed. A similar argument applies in
the case of uniform convergence.
  \item \textit{The general case:} Clearly,
  \[C_n-L\leq A_n-L\leq B_n-L\]
  for all $n\in\N$. Since by hypothesis $(C_n-L)$ and $(B_n-L)$
  converge uniformly (resp. strongly) to $0\in B(H)$, it follows by
  the foregoing part of the proof that $(A_n-L)$ too converges
  uniformly (resp. strongly) to $0\in B(H)$. Accordingly, $(A_n)$ converges
  uniformly (resp. strongly) to $L\in B(H)$, and this marks the end of the
  proof.
\end{enumerate}
\end{proof}

The following result is a consequence of the above theorem.

\begin{cor}\label{poopqsfgfgalmzre21451899lmgmgmgmgmgmgm}
Let $(A_n)$ be a sequence of self-adjoint operators. Then $A_n\to 0$
weakly (resp. in norm, strongly) if $|A_n|\to 0$ weakly (resp. in
norm, strongly).
\end{cor}

\begin{proof}
Write
\[-|A_n|\leq A_n\leq |A_n|\]
for all $n$. So, if $|A_n|\to 0$, then so is $A_n$.
\end{proof}

A simple consequence is derived from the previous theorem. This
generalizes a well known result about sequences of real numbers
(where the commutativity is guaranteed in that context). Besides,
this result has a certain interest as regards the map $(A,B)\mapsto
AB$.

\begin{pro}
Let $(A_n)$ be a sequence of self-adjoint operators, and let $(B_n)$
 be a sequence of positive operators. If $|A_n|\leq M$ for some positive $M\in B(H)$ and
all $n\in \N$,
 $B_n\to 0$ weakly (resp. strongly, uniformly), $A_nB_n=B_nA_n$ and $B_nM=MB_n$ for all $n$, then $A_nB_n\to
0$ weakly (resp. strongly, uniformly).
\end{pro}

\begin{proof}Since $|A_n|\leq M$ for all $n\in \N$, it ensues that
$-M\leq A_n\leq M$. Since $A_nB_n=B_nA_n$, $B_n\geq0$, and
$B_nM=MB_n$ for all $n$, it is seen that
\[-MB_n\leq A_nB_n\leq MB_n,\]
still for all $n\in\N$. Because $(-MB_n)$ and $(MB_n)$ both go to
the zero operator in either of the three topologies, Theorem
\ref{sandwich rule S.A. oper THM} now intervenes and yields
$A_nB_n\to 0$ with respect to each of the three topologies, and the
proof is therefore complete.
\end{proof}

\begin{cor}\label{OMANNNN}
Let $(A_n)$ be a sequence of self-adjoint operators, and let $(B_n)$
be a sequence of positive operators. If $|A_n|\leq \alpha I$ for
some $\alpha\geq 0$ and all $n\in \N$,
 $B_n\to 0$ weakly (resp. strongly, uniformly) and $A_nB_n=B_nA_n$ for all $n$, then $A_nB_n\to
0$ weakly (resp. strongly, uniformly).
\end{cor}

The following corollary shows the weak continuity of $(A,B)\mapsto
AB$ at one point (namely 0) when restricted to the set of positive
self-adjoint operators (in particular, $A\mapsto A^2$ is weakly
continuous when restricted to positive self-adjoint operators).

\begin{cor}\label{hichemmmmmmmm54458778kmdfgklmmlklmkfghùml}Let $(A_n)$ and $(B_n)$ be two sequences of positive operators. If $A_n,B_n\to 0$ weakly, and
$A_nB_n=B_nA_n$ for all $n$, then $A_nB_n\to 0$ weakly.
\end{cor}

\begin{proof}
Since $A_n\to 0$ weakly, the real sequence $\langle A_nx,x\rangle$
goes to 0. A simple argument allows us to write $\langle
A_nx,x\rangle\leq \alpha\langle x,x\rangle$ for some $\alpha\geq0$,
all $n\in\N$ and all $x\in H$. In other words, $0\leq A_n\leq \alpha
I$. Now, Corollary \ref{OMANNNN} gives the desired conclusion.
\end{proof}

\section{Counterexamples}

We give a couple of counterexamples. The first one shows the failure
of the conclusion in Corollary
\ref{hichemmmmmmmm54458778kmdfgklmmlklmkfghùml} when the positivity
is dropped.

\begin{exa}Let $S$ be the usual shift operator defined on
$\ell^2(\N)$, then set $A_n=S^n+(S^n)^*$. Each $A_n$ is self-adjoint
and $A_n\to 0$ weakly for it is well-known that $S^n$ converges to
$0$ weakly, as does $(S^n)^*$. Hence $A_n\to 0$ weakly. Since $S$ is
hyponormal, so is $S^n$ (recall that $T\in B(H)$ is hyponormal when
$TT^*\leq T^*T$). Therefore,
\[|A_n|\leq 2I\]
for all $n$ by say Exercise 12.3.11 in
\cite{Mortad-Oper-TH-BOOK-WSPC}.

Now, set $B_n=A_n$, and so $B_n\to 0$ weakly. Clearly
\[A_nB_n=A_n^2=(S^n)^2+[(S^n)^*]^2+I+S^n(S^n)^*.\]

Since it is easy to see that $(S^n)^2,[(S^n)^*]^2\to 0$ weakly, it
follows that $A_n^2\not\to0$ as otherwise $S^n(S^n)^*\to -I$, and
this is impossible because $S^n(S^n)^*$ is a sequence of positive
operators and so it cannot have a non-positive weak limit.
\end{exa}

\begin{rema}
This example may also be used to show the weak discontinuity of
$A\mapsto A^2$ over $B(H)$. It is simpler than the one which
appeared in Problem 114 in \cite{halmos-book-1982}. Indeed, by the
example above, $A_n\to 0$ weakly whilst $A_n^2\not\to 0$ weakly.
\end{rema}

Now, we give the second counterexample. It shows that the converse
of Corollary \ref{poopqsfgfgalmzre21451899lmgmgmgmgmgmgm} need not
be true. In general, if $A_n\to 0$ weakly, then $|A_n|$ does not
have to go 0 weakly. For example, consider $A_n=S^n$ where $S$ is
the shift operator on $\ell^2(\N)$. Then $A_n\to 0$ weakly but
$|A_n|=I$ for all $n$, and so $|A_n|\not\to 0$ weakly. A
counterexample is therefore more interesting in the class of
sequences of self-adjoint operators, and this is seen next.

\begin{exa}Let $S$ be the usual shift operator defined on
$\ell^2(\N)$, then set $A_n=S^n+(S^n)^*$. Each $A_n$ is self-adjoint
and $A_n\to 0$ weakly. We claim that $|A_n|$ does not converge
weakly to 0. For the sake of contradiction, suppose $|A_n|\to 0$. By
Corollary \ref{hichemmmmmmmm54458778kmdfgklmmlklmkfghùml} (by taking
$A_n=B_n$ for all $n$), it ensues that $|A_n|^2\to 0$ weakly.
However, we already know from the above example that
\[|A_n|^2=A_n^2=(S^n)^2+[(S^n)^*]^2+I+S^n(S^n)^*\not\longrightarrow 0\]
weakly.

\end{exa}

\bibliographystyle{amsplain}

\begin{thebibliography}{1}


\bibitem{Boucif-Dehimi-Mortad}
I. Boucif, S. Dehimi and M. H. Mortad. On the absolute value of
unbounded operators, \textit{J. Operator Theory}, \textbf{82/2}
(2019) 285-306.

\bibitem{Dehimi-Mortad-REID}
S. Dehimi, M. H. Mortad. Generalizations of Reid inequality,
\textit{Mathematica Slovaca}, \textbf{68/6} (2018) 1439-1446.

\bibitem{halmos-book-1982}
P. R. Halmos. \textit{ A Hilbert space problem book}, Springer, 1982
(2nd edition).

\bibitem{Kubrusly-book-operatopr-exercise-sol}
C. S. Kubrusly. \textit{Hilbert space operators, A problem solving
approach}, Birkh\"{a}user. Boston, Inc., Boston, MA, 2003.

\bibitem{Kubrusly-2011-COURS+EXO-FAT-BOOK}
C. S. Kubrusly. \textit{The elements of operator theory}.
Birkh\"{a}user/Springer, New York, 2011 (2nd edition).

\bibitem{Mortad-Oper-TH-BOOK-WSPC}
M. H. Mortad. \textit{An operator theory problem book}, World
Scientific Publishing Co., (2018).

\bibitem{mortad-Abs-value-BD}
M. H. Mortad. On the absolute value of the product and the sum of
linear operators, \textit{Rend. Circ. Mat. Palermo,  II. Ser},
\textbf{68/2} (2019) 247-257.

\bibitem{Mortad-counterexamples book OP TH}
M. H. Mortad, \textit{Counterexamples in operator theory}, (book, to
appear). Birkh\"{a}user/Springer.


\bibitem{RS1}
M. Reed, B. Simon. Methods of modern mathematical physics, Vol. {\bf
1}: \textit{Functional analysis}, Academic Press. 1972.

\bibitem{Simon-OPER-TH-BOOK-2015}
B. Simon. \textit{Operator theory. A Comprehensive Course in
Analysis}, Part \textbf{4}. American Mathematical Society,
Providence, RI, 2015.


\bibitem{Weidmann}
J. Weidmann. \textit{Linear Operators in Hilbert Spaces}, Springer,
1980.

\end{thebibliography}

\end{document}